\documentclass[reqno]{amsart}

\usepackage{enumitem}
\usepackage{amssymb,amsmath,mathtools}
\usepackage{microtype}
\usepackage{bbm}
\usepackage{hyperref}

\newcommand*{\mailto}[1]{\href{mailto:#1}{\nolinkurl{#1}}}
\newcommand{\arxiv}[1]{\href{http://arxiv.org/abs/#1}{arXiv:#1}}

\newcommand{\msc}[1]{\href{http://www.ams.org/msc/msc2010.html?t=&s=#1}{#1}}

\newif\iflong
\longtrue

\newtheorem{theorem}{Theorem}[section]
\newtheorem{lemma}[theorem]{Lemma}
\newtheorem{corollary}[theorem]{Corollary}
\newtheorem{remark}[theorem]{Remark}
\theoremstyle{definition}

\newtheorem{conjecture}{Conjecture}[section]

\newcommand{\R}{{\mathbb R}}
\newcommand{\N}{{\mathbb N}}

\newcommand{\C}{{\mathbb C}}

\newcommand{\cH}{{\mathcal H}}
\newcommand{\cG}{{\mathcal G}}

\newcommand{\id}{{\mathbbm{1}}}
\newcommand{\OO}{{\mathcal O}}

\newcommand{\SU}{\mathrm{SU}}
\newcommand{\GL}{\mathrm{GL}}
\newcommand{\slt}{\mathfrak{sl}(2,\R)}
\newcommand{\SLt}{\mathrm{SL}(2,\R)}

\newcommand{\p}{{\bf p}}
\newcommand{\g}{{\bf g}}

\newcommand{\I}{\mathrm{i}}
\newcommand{\E}{\mathrm{e}}

\newcommand{\nn}{\nonumber}
\newcommand{\be}{\begin{equation}}
\newcommand{\ee}{\end{equation}}

\newcommand{\ti}{\tilde}
\newcommand{\abs}[1]{\left\lvert #1 \right\rvert}

\newcommand{\norm}[1]{\left\lVert #1 \right\rVert}

\newcommand{\floor}[1]{\lfloor #1 \rfloor}

\newcommand{\lam}{\lambda}

\newcommand{\al}{\alpha}

\newcommand{\cB}{\mathcal{B}}

\newcommand\half{\frac12}
\newcommand\thalf{\tfrac12}

\newcommand{\hyp}[5]{\,\mbox{}_{#1}F_{#2}\!\left(
  \genfrac{}{}{0pt}{}{#3}{#4};#5\right)}

\makeatletter
\newcommand{\dlmf}[1]{%
\cite[%
 \def\nextitem{\def\nextitem{, }}%
 \@for \el:=#1\do{\nextitem\expandafter\dlmf@eq@href\el...\end}%
]{dlmf}%
}
\def\dlmf@eq@href#1.#2.#3.#4\end{%
  \href{http://dlmf.nist.gov/#1.#2.E#3}{(#1.#2.#3)}}
\makeatother

\numberwithin{equation}{section}

\begin{document}

\title[Bernstein-type Inequalities and Dispersion Estimates]{Jacobi Polynomials, Bernstein-type Inequalities and Dispersion Estimates for the Discrete Laguerre Operator}

\author[T. Koornwinder]{Tom Koornwinder}
\address{Korteweg-de Vries Institute for Mathematics\\ University of Amsterdam\\
1090 GE Amsterdam\\ Netherlands}
\email{\mailto{T.H.Koornwinder@uva.nl}}
\urladdr{\url{https://staff.fnwi.uva.nl/t.h.koornwinder/}}

\author[A. Kostenko]{Aleksey Kostenko}
\address{Faculty of Mathematics and Physics\\ University of Ljubljana\\ Jadranska ul.\ 19\\ 1000 Ljubljana\\ Slovenia\\ and Faculty of Mathematics\\ University of Vienna\\ 
Oskar-Morgenstern-Platz 1\\ 1090 Vienna\\ Austria}
\email{\mailto{Aleksey.Kostenko@fmf.uni-lj.si};\ \mailto{Oleksiy.Kostenko@univie.ac.at}}
\urladdr{\url{http://www.mat.univie.ac.at/~kostenko/}}

\author[G. Teschl]{Gerald Teschl}
\address{Faculty of Mathematics\\ University of Vienna\\
Oskar-Morgenstern-Platz 1\\ 1090 Wien\\ Austria\\ and International 
Erwin Schr\"odinger Institute for Mathematical Physics\\ 
Boltzmanngasse 9\\ 1090 Wien\\ Austria}
\email{\mailto{Gerald.Teschl@univie.ac.at}}
\urladdr{\url{http://www.mat.univie.ac.at/~gerald/}}

\thanks{{\it Research supported by the Austrian Science Fund (FWF) under Grant No.\ P26060}}
\thanks{Adv. Math. {\bf 333}, 796--821 (2018)}

\keywords{Schr\"odinger equation, dispersive estimates, Jacobi polynomials}
\subjclass[2010]{Primary \msc{33C45}, \msc{47B36}; Secondary \msc{81U30}, \msc{81Q05}}

\begin{abstract}
The present paper is about Bernstein-type estimates for Jacobi polynomials and their applications to various branches in mathematics.
This is an old topic but we want to add a new wrinkle by establishing some intriguing connections with dispersive estimates for a certain class
of Schr\"odinger equations whose Hamiltonian is given by the generalized Laguerre operator. More precisely, we show that dispersive estimates
for the Schr\"odinger equation associated with the generalized Laguerre operator
are connected with Bernstein-type inequalities for Jacobi polynomials. We use known uniform estimates for Jacobi polynomials
to establish some new dispersive estimates. In turn, the optimal dispersive decay estimates lead to new Bernstein-type inequalities.
\end{abstract}

\maketitle

\section{Introduction}
To set the stage, for $\alpha$, $\beta>-1$, let
$w^{(\alpha,\beta)}(x)= (1-x)^\alpha (1+x)^\beta$ for $x\in (-1,1)$
be a Jacobi weight.
The corresponding orthogonal polynomials $P_n^{(\alpha,\beta)}$,
normalized by
\be\label{eq:normal}
P_n^{(\alpha,\beta)} (1) = \binom{n+\alpha}{n} = \frac{(\alpha+1)_n}{n!}
\ee
for all $n\in\N_0$ (see \eqref{K6} for notation of Pochhammer symbols
and binomial coefficients), are called the {\em Jacobi polynomials}.
They are expressed as (terminating) Gauss hypergeometric
series \eqref{K7}
by \cite[(4.21.2)]{sz}
\be\label{eq:01}
\frac{P_n^{(\alpha,\beta)}(x)}{P_n^{(\alpha,\beta)}(1)} =
\hyp21{-n,n+\alpha+\beta+1}{\alpha+1}{\frac{1-x}{2}}.
\ee
They also satisfy Rodrigues' formula \cite[(4.3.1), (4.3.2)]{sz}
\begin{align}
P_n^{(\alpha,\beta)}(x) &=\sum_{k=0}^n \binom{n+\alpha}{n-k} \binom{n+\beta}{k} \left(\frac{x-1}{2}\right)^k\left(\frac{x+1}{2}\right)^{n-k}\label{K5}\\[1mm]
&= \frac{(-1)^n}{2^n n!}(1-x)^{-\alpha}(1+x)^{-\beta} \frac{d^n}{dx^n} \big\{(1-x)^{\alpha+n}(1+x)^{\beta+n} \big\}.\label{eq:01a}
\end{align}
Note that, by \eqref{K5}, $P_n^{(\alpha,\beta)}(x)$ is for given $n$
a polynomial in $x$, $\alpha$ and $\beta$. Thus, if we don't need
the orthogonality relations of the Jacobi polynomials, then we are not 
restricted by the bounds $\alpha,\beta>-1$.

The (squared normalized)
$L^2$ norm of $P_n^{(\alpha,\beta)}$ is given by \cite[(4.3.3)]{sz}
\begin{multline}\label{eq:02}
\frac{\Gamma(\alpha+\beta+2)}
{2^{\alpha+\beta+1}\Gamma(\alpha+1)\Gamma(\beta+1)}
\int_{-1}^1 |P_n^{(\alpha,\beta)}(x)|^2 w^{(\alpha,\beta)}(x)dx\\
= \frac{n+\alpha+\beta+1}{2n+\alpha+\beta+1}\,
\frac{(\alpha+1)_n(\beta+1)_n}{(\alpha+\beta+2)_n\,n!}\,.
\end{multline}
Jacobi polynomials include the ultraspherical (Gegenbauer) polynomials
\cite[(4.37.1)]{sz}
\be\label{eq:gegenpol}
P_n^{(\lambda)}(x) := \frac{(2\lambda)_n}{(\lambda+1/2)_n}\,
P_n^{(\lambda-\frac{1}{2},\lambda-\frac{1}{2})}(x),
\ee
where $\lambda>-1/2$ (for $\lambda=0$, one needs to replace \eqref{eq:gegenpol} by a suitable limit, see \cite[Eq. (4.7.8)]{sz}),  and the Legendre polynomials
\be\label{eq:legenpol}
P_n(x) := P_n^{(1/2)}(x) = P_n^{(0,0)}(x) =
\frac{1}{2^nn!}\frac{d^n}{dx^n} (x^2-1)^{n}.
\ee
We shall denote the corresponding orthonormal polynomials by ${\p}_n^{(\alpha,\beta)}$ for Jacobi, $\p^{(\lambda)}_n$ for Gegenbauer, and $\p_n$ for Legendre polynomials.
 
The Rodrigues formula \eqref{eq:01a} immediately implies
\be\label{eq:symmetry}
P_n^{(\alpha,\beta)}(-x) = (-1)^nP_n^{(\beta,\alpha)}(x),
\ee
and hence 
\be\label{eq:normal_b}
P_n^{(\alpha,\beta)} (-1) =(-1)^n \binom{n+\beta}{n} =
(-1)^n\,\frac{(\beta+1)_n}{n!}\,.
\ee
It is well known that the absolute value of $P_n^{(\alpha,\beta)}$ attains its maximum at the endpoints of the interval $[-1,1]$ 
\be\label{eq:Cunif}
\max_{x\in [-1,1]}\big|P_n^{(\alpha,\beta)}(x)\big| = \max_{x\in\{-1,1\}} \big|P_n^{(\alpha,\beta)}(x)\big|
=\binom{n+\max(\alpha,\beta)}{n},
\ee
if $\max(\alpha,\beta)\ge -1/2$ (see \cite[Theorem 7.32.1]{sz}). 

The asymptotic behavior of Jacobi polynomials for large $n$ is rather well understood (see, e.g., \cite[Chapter VIII]{sz}), however, almost all these formulas are not uniform in $\alpha$ and $\beta$.
The main focus of the present paper is on uniform estimates for
\be\label{eq:Bern_ab}
(1-x)^{a}(1+x)^{b} \big|P_n^{(\alpha,\beta)}(x)\big|
\ee
on the whole segment of orthogonality $[-1,1]$ with some $a\ge 0$ and $b\ge 0$ (which might depend on $\alpha$ and $\beta$).
Historically, the first result of this type is Bernstein's inequality\footnote{In order to avoid confusions with the Bernstein inequality for (algebraic) polynomials in the unit disk ($\max_{|z|\le 1}|P'(z)| \le n\cdot\max_{|z|\le 1}|P(z)|$, where $n$ is the degree of $P$), throughout the text ``Bernstein's inequality'' should read as ``Bernstein's inequality for Legendre/Gegenbauer/Jacobi polynomials" meaning the uniform (weighted) estimate for the corresponding family of orthogonal polynomials.} for the Legendre polynomials (\cite[Theorem 7.3.3]{sz})
\be\label{eq:Bernstein}
(1-x^2)^{1/4} |P_n(x)| \le \frac{2}{\sqrt{\pi(2n+1)}}, \quad x\in[-1,1],
\ee  
(the refined version \eqref{eq:Bernstein} was proved in \cite{ah}, see also \cite{lo1}). The constant $\sqrt{2/\pi}$ in \eqref{eq:Bernstein} is sharp. Moreover (see \cite[Theorem 12.1.6]{sz}), the following expression 
\be\label{eq:szego}
(1-x^2)^{1/4} \sqrt{w^{(\alpha,\beta)}(x)}\,\p_n^{(\alpha,\beta)}(x) 
\ee
asymptotically equioscillates between $-\sqrt{2/\pi}$ and $\sqrt{2/\pi}$ when $n$ tends to infinity (the latter holds for a wider class of orthonormal polynomials) and hence a lot of effort has been put in proving the estimates for \eqref{eq:Bern_ab} with $a=\frac{\alpha}{2}+\frac{1}{4}$ and $b=\frac{\beta}{2}+\frac{1}{4}$. Thus, for ultraspherical polynomials the corresponding estimates can be found in \cite[Theorem 7.33.2]{sz} (the case $\lambda\in (0,1)$, see also \cite{lo2} for a refinement), \cite{loh} (the case $\lambda>0$) and  \cite{foe} (the case $\lambda\ge 1$). In the nonsymmetric case, let us mention \cite{cgw}, \cite{emn} and the recent papers \cite{haa}, \cite{kra07}, \cite{kra08}. 
Let us also mention that it was conjectured by Erd\'elyi, Magnus and Nevai \cite{emn} that 
\be\label{eq:emn_con}
\max_{x\in(-1,1)} (1-x^2)^{1/4} \sqrt{w^{(\alpha,\beta)}(x)}\,
|\p_n^{(\alpha,\beta)}(x)| \le C \max(1,(|\alpha|+|\beta|)^{1/4})
\ee
for all $n\in\N_0$ and $\alpha$, $\beta\ge -1/2$. Notice that a weaker bound $\OO(\max(1,(\alpha^2+\beta^2)^{1/4}))$ was proved in \cite[Theorem 1]{emn}. On the other hand, the Erd\'elyi--Magnus--Nevai conjecture \eqref{eq:emn_con} was confirmed for all $n\in\N_0$ and $\alpha$, $\beta\in (-1/2,1/2)$ in \cite{cgw} (with a sharp estimate of the error term, see also \cite{ga}) and for all $n\ge 6$ and $\alpha$, $\beta\ge (1+\sqrt{2})/4$ in \cite{kra07}, \cite{kra08} (see also \cite{haa}).

The estimates for \eqref{eq:Bern_ab} with $a\neq\frac{\alpha}{2}+\frac{1}{4}$ and $b\neq\frac{\beta}{2}+\frac{1}{4}$ are much less studied, however, they are important in many applications. Let us mention only a few of them. First of all, ultraspherical polynomials arise in quantum mechanics as spherical harmonics. More precisely, the $L^2$ normalized spherical harmonics, which are eigenfunctions of the Laplace--Beltrami operator on the sphere $\mathbb{S}^2$, are given by  (cf. \cite[(4.7.35)]{sz})
\be\label{eq:harmonics}
\begin{split}
Y_l^m(\theta,\varphi) :=  \frac{(-1)^m}{\sqrt{2\pi}}\,\E^{\I m\varphi}
\sin^m(\theta)\,\p_{l-m}^{(m+1/2)}(\cos(\theta)),
\end{split}
\ee
if $m\in \{0,\dots, l\}$. 
Therefore, \eqref{eq:Bern_ab} provides uniform weighted $L^\infty$ estimates on eigenfunctions of the Laplace--Beltrami operator on $\mathbb{S}^2$. In particular, the following inequality was established in \cite[Theorem 1]{bdwz}:
\be\label{eq:burqest}
\max_{x\in (-1,1)} |x|^{1/6}(1-x^2)^{m/2+1/6} \big|\p_n^{(m+1/2)}(x)\big|\le C\, (n+m+1)^{1/6},
\ee
with some $C>0$, which does not depend on $n$, $m\in\N_0$. Moreover, \eqref{eq:burqest} and Krasikov's estimates \cite{kra08} were employed in \cite{bdwz} and \cite{rw}, respectively, in order to obtain bounds on the number of samples necessary for recovering sparse eigenfunction expansions on surfaces of revolution. 

The next example is also widely known. More precisely, Jacobi polynomials appear as coefficients of the so-called Wigner $d$-matrix 
(see Theorem \ref{thm:d-matrix}). Thus Bernstein-type estimates imply uniform bounds on a complete set of matrix coefficients for irreducible representations of ${\rm SU}(2)$
(see \cite{haa} and Section \ref{sec:II} below).  Furthermore, these inequalities play a very important role in the study of simple Lie groups. Namely, the Bernstein inequality and the Haagerup--Schlichtkrull inequality (see \eqref{eq:g_unif3} below)  were used in \cite{ldl} and \cite{hdl}, \cite{hdl2}, respectively, to study the approximation property of Haagerup and Kraus \cite{hk} for connected simple Lie groups. 

Finally, our interest in the estimates of the type \eqref{eq:Bern_ab} comes from the so-called dispersive estimates for discrete Laguerre operators\be\label{eq:H0}
H_\alpha := \begin{pmatrix} 
1+\alpha & \sqrt{1+\alpha} & 0 &  \cdots \\[1mm]
\sqrt{1+\alpha} & 3 + \alpha & \sqrt{2(2+\alpha)}   & \cdots \\[1mm]
0 & \sqrt{2(2+\alpha)} & 5 + \alpha  &  \cdots \\[1mm]
\vdots&\vdots&\vdots&\ddots
\end{pmatrix},\quad \alpha>-1,
\ee
acting in $\ell^2(\N_0)$. Explicitly, $H_\alpha = \big(h^{(\alpha)}_{n,m}\big)_{n,m\in\N_0}$ with $h^{(\alpha)}_{n,m}=0$ if $|n-m|>1$ and 
\begin{align*}
h^{(\alpha)}_{n,n} = 2n+1 + \alpha,\quad  h^{(\alpha)}_{n,n+1} = h_{n+1,n}^{(\alpha)}= \sqrt{(n+1)(n+1+\alpha)},\quad n\in \N_0.
\end{align*}
It is a special case of a self-adjoint Jacobi operator whose generalized eigenfunctions are precisely the Laguerre polynomials $L_n^{(\alpha)}$,
explaining our name for \eqref{eq:H0}.

The operator $H_\alpha$ features prominently in the recent study of nonlinear waves in $(2+1)$-dimensional noncommutative scalar field theory \cite{a06, a13, gms}. The coefficient $\alpha$ in \eqref{eq:H0} can be seen as a measure of the delocalization of the field configuration and it is related to the planar angular momentum \cite{a13}. In particular, $\alpha=0$ corresponds to spherically symmetric waves and it has attracted further interest in \cite{cfw03,ks15a,ks15b,ks15c}. 
 As this operator appears as the linear part in the nonlinear Klein--Gordon equation  \cite{cfw03} and the nonlinear Schr\"odinger equation
\be\label{eq:nls}
\I \dot \psi(t,n) = H_0 \psi(t,n) - |\psi(t,n)|^{2\sigma}\psi(t,n),\quad \sigma\in\N,\quad (t,n)\in \R_+\times \N_0,
\ee
investigated in the recent work of Krueger and Soffer \cite{ks15a,ks15b,ks15c}, dispersive estimates play a crucial role in the understanding of stability of the soliton manifolds appearing in these models (for further details see \cite{cfw03,gms,ks15a,ks15b,ks15c}). It turns out that the required dispersive decay estimates for the evolution group $\E^{-\I tH_\alpha}$ lead to Bernstein-type estimates for \eqref{eq:Bern_ab} (see \cite{kt} and Sections \ref{sec:V}--\ref{sec:VII} below).  
All these connections are mathematically very appealing and we hope that this note will stimulate further research in this direction.

In conclusion let us briefly outline the content of the paper. In the next section we introduce discrete Laguerre operators and briefly review their spectral properties. In Section \ref{sec:IV}, we present a connection between discrete Laguerre operators $H_\alpha$ and Jacobi polynomials. More precisely, we show that the kernel of the evolution group $\E^{\I tH_\alpha}$ can be expressed by means of Jacobi polynomials (Theorem \ref{thm:explicit}). This result establishes a connection between uniform estimates for \eqref{eq:Bern_ab} and dispersive estimates for the evolution group $\E^{\I tH_\alpha}$. In Section \ref{sec:II}, we review the connection between irreducible representations of $\SU(2)$ and Jacobi polynomials. The latter, in particular, implies the estimates for \eqref{eq:Bern_ab} with $a=\frac{\alpha}{2}$ and $b=\frac{\beta}{2}$ when $\alpha$, $\beta\in\N_0$ (see \eqref{eq:g_unif1} and \eqref{eq:g_unif2}). In Section \ref{sec:new}, we prove the following Bernstein-type estimate 
\be\label{eq:conj01}
\left(\frac{1+x}{2}\right)^{{\beta}/{2}} \Big|P_n^{(\alpha,\beta)}\left(x\right)\Big| \le \binom{n+\alpha}{n},\quad x\in [-1,1],
\ee
if $\beta\ge 0$ and $\alpha\ge \beta - \floor{\beta}$. 

Finally, Bernstein-type inequalities enable us to prove the decay estimates for the evolution group $\E^{-\I tH_\alpha}$, 
which we discuss in Section \ref{sec:V}.  First of all, using the known Bernstein-type inequalities, we prove the decay estimates of order $\OO(t^{-1})$ (Theorem~\ref{thm:decay1}) and $\OO(t^{-1/2})$  if $\alpha\ge 0$, however, with a better behavior of weights~$\sigma$
(Theorem~\ref{thm:decay2}). On the other hand, the new inequality \eqref{eq:conj01} enables us to show that 
 \be\label{eq:I_decay}
 \|\E^{- \I tH_\alpha}\|_{\ell^1(\sigma_\alpha)\to \ell^\infty(\sigma_{\alpha}^{-1})} =\left(\frac{1}{1+t^2}\right)^{\frac{1+\alpha}{2}}, \quad t\in\R,
 \ee
for all $\alpha\ge 0$, with the weights $\sigma_\alpha$ given by $\sigma_\alpha=\{\binom{n+\alpha}{n}^{1/2}\}_{n\ge 0}$ (see Theorem \ref{thm:decay_a}). We finish our paper with some further comments on new Bernstein-type inequalities and certain parallels between dispersive estimates for discrete Laguerre operators and one-dimensional spherical Schr\"odinger operators (see Section~\ref{sec:VII}).

\subsection*{Notation}
$\R$ and $\C$ have the usual meaning. Also write $\R_+ := (0,\infty)$,
$\N:=\{1,2,\ldots\}$ and $\N_0:=\N\cup\{0\}$. By
$\Gamma$ is denoted the classical gamma function \dlmf{5.2.1}.
For $x\in\C$ and $n\in\N_0$
\be\label{K6}
(x)_n:=\begin{cases} x(x+1)\cdots(x+n-1),& n\in\N\\ 1, & n=0 \end{cases};\qquad
\binom{n+x}{n} := \frac{(x+1)_n}{n!}
\ee
denote the {\em Pochhammer symbol} \dlmf{5.2.4}
and the {\em binomial coefficient}, respectively. Notice that for $-x\notin \N_0$
\[
(x)_n = \frac{\Gamma(x+n)}{\Gamma(x)},\qquad 
\binom{n+x}{n} = \frac{\Gamma(x+n+1)}{\Gamma(x+1)\Gamma(n+1)}\,.
\]
Moreover, the above formulas allow to define the Pochhammer symbol and the binomial coefficient for noninteger $x$, $n>0$. 
Finally, for $-c \notin \N_0$
the
{\em Gauss hypergeometric function} \dlmf{15.2.1} is defined by 
\be\label{K7}
\hyp21{a,b}{c}{z} := \sum_{k=0}^\infty \frac{(a)_k(b)_k}{(c)_k k!}\,z^k
\quad\mbox{($|z|<1$ or else $-a$ or $-b\in \N_0$).}
\ee

\section{Spectral properties of the discrete Laguerre operators}\label{sec:III} 

We start with a precise definition of the operator $H_\alpha$ associated with the Jacobi matrix \eqref{eq:H0}. For a sequence $u=\{u_n\}_{n\ge 0}$ we define the difference expression $\tau_\alpha\colon u\mapsto \tau_\alpha u$ by setting
\be\label{eq:tau}
(\tau_\alpha u)_n : = \sqrt{n(n+\alpha)}\,u_{n-1} + (2n+1 +\alpha)u_n +
\sqrt{(n+1)(n+1+\alpha)}\,u_{n+1},
\ee
where $u_{-1}:=0$ for notational convenience.
Then the operator $H_\alpha$ associated with the Jacobi matrix \eqref{eq:H0} is defined by
\begin{align}\begin{split}
H_\alpha\colon \begin{array}[t]{lcl} \mathcal{D}_{\max} &\to& \ell^2(\N_0) \\ u &\mapsto& \tau_\alpha u\  \end{array},
\end{split}\end{align}
where $ \mathcal{D}_{\max} = \{u\in \ell^2(\N_0)|\, \tau_\alpha u\in \ell^2(\N_0)\}$. 
The spectral properties of $H_\alpha$ are well known.
For the sake of completeness
we collect them in the following theorem and give a short proof.
\begin{theorem}\label{thm:spH_0}
Let $\alpha>-1$. Then:
 \begin{enumerate}[label=(\roman*), ref=(\roman*), leftmargin=*, widest=iii] 
\item The operator $H_\alpha$ is a positive self-adjoint operator.
\item The spectrum of $H_\alpha$ is purely absolutely continuous and coincides with $[0,\infty)$.
\item The Weyl function and the corresponding spectral measure are given by
\begin{align}\nn
m_\alpha(z)  &= \frac{1}{\Gamma(\alpha+1)}\int_0^{+\infty} \frac{\E^{-\lambda}\lambda^\alpha}{\lambda-z}d\lambda=
\E^{-z} { E}_{1+\alpha}(-z) ,\quad z\in\C\setminus [0,\infty),\\
d\rho_\alpha(\lambda) &= \frac{\id_{\R_+}(\lambda)}{\Gamma(\alpha+1)}\E^{-\lambda}\lambda^\alpha d\lambda,\quad \lambda\in\R,
\end{align}
where $E_p(z) := z^{p-1} \int_z^\infty \E^{-t} t^{-p} dt$ denotes the principal value of the generalized exponential integral \dlmf{8.19.2}.
\end{enumerate}
\end{theorem}

\begin{proof}
(i) Self-adjointness clearly follows from the Carleman test (see, e.g., \cite{akh}, \cite[(2.165)]{tjac}).
Nonnegativity as well as item (ii) immediately follow from (iii), so let us prove (iii).
Notice that the orthogonal polynomials for $H_\alpha$ are given by
\be
P_{\alpha,n}(z) = \frac{(-1)^{n}}{\sigma_{\alpha}(n)}L_n^{(\alpha)}(z),\quad n\in \N_0,
\ee
where 
\be\label{eq:sigma_a}
\sigma_{\alpha}(n) = \sqrt{L_n^{(\alpha)}(1)} = \binom{n+\alpha}{n}^{1/2},
\ee
and $L_n^{(\alpha)}$ are the Laguerre polynomials \cite[Section 5.1]{sz}
\be\label{eq:laguerpol}
\frac{L_n^{(\alpha)}(z)}{L_n^{(\alpha)}(1)}
=\hyp11{-n}{\alpha+1}z
=\sum_{k=0}^n\frac{(-n)_k}{(\alpha+1)_k\,k!}\,z^k.
\ee
The recurrence formula for the Laguerre polynomials \cite[(5.1.10)]{sz} implies that 
$u:=\{P_{\alpha,n}(z)\}_{n\in\N_0}$ satisfies $(\tau_\alpha u)_n = zu_n$ for all $n\ge 0$. 
Furthermore, the polynomials $L_n^{(\alpha)}$
satisfy the orthogonality relations \cite[(5.1.1)]{sz}
\be\label{eq:lagorth}
\frac1{\Gamma(\alpha+1)}
\int_{0}^{\infty} L_n^{(\alpha)}(\lambda) L_k^{(\alpha)}(\lambda)
\E^{-\lambda}\lambda^\alpha\, d\lambda = \binom{n+\alpha}{n}\,
\delta_{nk},\quad n,k\in\N_0.
\ee
Therefore, \eqref{eq:lagorth} and (i) imply that $d\rho_\alpha$ is the spectral measure of $H_\alpha$, that is, $H_\alpha$ is unitarily equivalent to a multiplication operator in $L^2(\R_+,d\rho_\alpha)$. 
It remains to note that the corresponding Weyl function is the Stieltjes transform of the measure $d\rho_\alpha$ (cf.\ e.g.\ \cite[Chapter~2]{tjac}). 
\end{proof}



\begin{remark}\label{rm:sl(2)_1}
The operator $H_\alpha$, when restricted to $\ell^2_c(\N_0)$, 
can be seen as occurring in a discrete series representation of the Lie algebra $\slt$.
First define operators $A,X,Y$ on this linear span by
\[
Au_n:=(2n+\alpha+1)u_n,\;
Xu_n:=\sqrt{(n+1)(n+\alpha+1)}\,u_{n+1},\;
Yu_n:=\sqrt{n(n+\alpha)}\,u_{n-1}.
\]
They satisfy the commutator relations
\[
[A,X]=2X,\quad
[A,Y]=-2Y,\quad
[X,Y]=-A.
\]
Now consider the skew-hermitian operators
\[
J_0:=X-Y,\quad
J_+:=\thalf \I(-A+X+Y),\quad
J_-:=\thalf \I(A+X+Y)=\thalf\I H_\alpha.
\]
They form an $\slt$ triple:
\[
[J_0,J_+]=2J_+,\quad
[J_0,J_-]=-2J_-,\quad
[J_+,J_-]=J_0.
\]
Thus we have a representation of the Lie algebra $\slt$:
\[
\begin{pmatrix}1&0\\0&-1\end{pmatrix}\to J_0,\quad
\begin{pmatrix}0&1\\0&0\end{pmatrix}\to J_+,\quad
\begin{pmatrix}0&0\\1&0\end{pmatrix}\to J_-.
\]
In particular, $\big(\begin{smallmatrix}0&-1\\1&0\end{smallmatrix}\big)$, which spans the
Lie subalgebra of the subgroup $K:=$SO(2) of $\slt$, is mapped in this
representation to $\I A$. If we compare with \cite[Section 7]{ko88},
which builds on \cite[Section 3]{bw}, we see that this representation,
when exponentiated to a unitary representation of
the Lie group $\SLt$, is a so-called discrete series representation $D_{\half(\alpha+1)}^+$ of $\SLt$ for $\alpha\in\N_0$, and otherwise,
for real $\alpha>-\thalf$, a similar representation of the universal
covering group of $\SLt$ (see \cite{sa}).
\end{remark}

\section{The evolution group $\E^{-\I tH_\alpha}$}\label{sec:IV}
In this and the following sections we look at the one-dimensional discrete Schr\"odinger equation
\begin{equation} \label{Schr}
  \I \dot \psi(t,n) = H_\alpha \psi(t,n), \quad  
  (t,x)\in\R\times \N_0,
\end{equation}
associated with the Laguerre operator $H_\alpha$ defined in the previous section. 
We begin by establishing a connection between the discrete Laguerre operators and Jacobi polynomials, which follows from the fact that the Laplace transform of a product of two Laguerre polynomials is expressed by means of a terminating Gauss hypergeometric series.

\begin{theorem}\label{thm:explicit}
Let $\alpha>-1$. The kernel\footnote{In analogy with the (integral) kernel of an integral operator we speak
about the (summation) kernel of a summation operator acting by a matrix.} of the operator $\E^{-\I tH_\alpha}$ is given by 
\begin{multline}\label{eq:explicit}
\E^{-\I tH_\alpha}(n,m)=\E^{-\I tH_\alpha}(m,n)\\
= \frac {1}{(1+\I t)^{1+\alpha}}\left(\frac{t+\I}{t-\I}\right)^n\left(\frac{t}{t-\I}\right)^{m-n}\,
\frac{\sigma_\alpha(m)}{\sigma_\alpha(n)}\ 
P_n^{(\alpha,m-n)}\left( \frac{t^2-1}{t^2+1}\right)
\end{multline} 
for all $n$, $m\in\N_0$.
\end{theorem} 

\begin{proof}
Similar to the case $\alpha=0$ (see \cite{kt}), one gets by employing Stone's formula (cf., e.g. \cite[\S 4.1]{tschroe})
\begin{multline}\label{PP}
\E^{-\I tH_\alpha}(n,m)
	=\frac{(-1)^{n+m}}{\sigma_\alpha(n)\sigma_\alpha(m)\Gamma(\alpha+1)}
\int_{0}^{\infty}\E^{-\I t \lambda} L_n^{(\alpha)}(\lambda) L_m^{(\alpha)}(\lambda) \E^{-\lambda}\lambda^\alpha\,d\lambda 
\end{multline}
for all $n$, $m\in\N_0$. 
It follows from \eqref{PP} that every element of the kernel of the operator $\E^{-\I tH_\alpha}$ is the Laplace transform of a product of two Laguerre polynomials. 
Then using \cite[(4.11.35)]{erd} and \dlmf{15.8.7} together with Euler's transformation \dlmf{15.8.1},
after lengthy but straightforward calculations one arrives at \eqref{eq:explicit}.
\end{proof}

\begin{remark}\label{rm:Meixner}
It is interesting to mention that the unitarity of  $\E^{-\I tH_\alpha}$ is equivalent to the
orthogonality relations for the Meixner polynomials \dlmf{18.20.7}
\be\label{K11}
M_n(x;\beta,c):=\hyp21{-n,-x}\beta{1-c^{-1}}.
\ee
Namely, 
\be\label{K4}
M_n(x;\beta,c)=\frac{n!}{c^n(\beta)_n}\,P_n^{(\beta-1,x-n)}(2c-1),
\ee
and then equation \eqref{eq:explicit} reads
\be\label{K2}
\frac{\E^{-\I tH_\alpha}(n,m)}{{\sigma_\alpha(m)}{\sigma_\alpha(n)}}=
\frac {1}{(1+\I t)^{1+\alpha}}\left(\frac{-\I t}{1+\I t}\right)^{n+m}
\  M_n\left(m;\alpha+1,\frac{t^2}{1+t^2}\right).
\ee
It remains to note that the orthogonality relations are \cite[Table \href{http://dlmf.nist.gov/18.19.T1}{18.19.1}]{dlmf} (with positive weights if $\beta>0$ and $0<c<1$)
\be\label{K3}
(1-c)^\beta\sum_{x=0}^\infty \frac{(\beta)_x c^x}{x!}\,
M_n(x;\beta,c)M_k(x;\beta,c)=
\frac{n!}{(\beta)_n c^n}\,\delta_{nk}.
\ee
\end{remark}

\begin{remark}
We continue Remark \ref{rm:sl(2)_1} and assume, for convenience, that
$\alpha\in\N_0$, so that we can refer to \cite[Section 7]{ko88}.
In the realization of the discrete series representation given there,
a $K$-basis \cite[(7.16)]{ko88} is given in terms of Laguerre polynomials
and the $K$-$K$ matrix elements \cite[(7.20)]{ko88} are in terms of Meixner
polynomials. This provides a further explanation of the observations in
Remark \ref{rm:Meixner}.
\end{remark}

The next result provides recurrence relations for the kernel of
$\E^{\I tH_\alpha}$.

\begin{corollary}\label{cor:recur}
Let $\alpha>-1$ and $n\le m$. Then 
\begin{align}
\E^{-\I tH_\alpha}(n+1,m+1) =
&\sqrt{\frac{(m+1)(m+1+\alpha)}{(n+1)(n+1+\alpha)}}\,
\frac{\I+t}{\I - t}\,\E^{-\I tH_\alpha}(n,m)\nn \\
&\quad +\frac{n+m+\alpha+2}{\sqrt{(n+1)(n+1+\alpha)}}\,
\frac{t}{\I - t}\,\E^{-\I tH_\alpha}(n,m+1)\\
 =& \frac{n+m+\alpha+2}{\sqrt{(n+1)(m+1)}}\,
 \frac{1}{1+\I t}\,\E^{-\I tH_{\alpha+1}}(n,m)\nn \\
&\quad+\sqrt{\frac{(n+\alpha+1)(m+1+\alpha)}{(n+1)(m+1)}}\,
\frac{\I+t}{\I - t}\,\E^{-\I tH_\alpha}(n,m).
\end{align}
\end{corollary}

\begin{proof}
Using the recurrence relations for Jacobi polynomials (see \cite[(4.5.4)]{sz}):
\begin{align*}
P_{n+1}^{(\alpha,\beta)}(x) &= \frac{n+\alpha+1}{n+1}P_{n}^{(\alpha,\beta)}(x) - \frac{2n+\alpha+\beta+2}{n+1}\frac{1-x}{2} P_{n}^{(\alpha+1,\beta)}(x)\\
&=  \frac{2n+\alpha+\beta+2}{n+1}\frac{1+x}{2} P_{n}^{(\alpha,\beta+1)}(x) - \frac{n+\beta+1}{n+1}P_{n}^{(\alpha,\beta)}(x),
\end{align*}
straightforward calculations complete the proof.
\end{proof}

We collect some special cases explicitly for later use.

\begin{corollary}\label{cor:cases}
 \begin{enumerate}[label=(\roman*), ref=(\roman*), leftmargin=*, widest=iii] 
\item In the case $n=0$ we have 
\be\label{eq:erd1}
\E^{-\I tH_\alpha}(0,m)  = 
\frac {1}{(1+\I t)^{1+\alpha}}\left(\frac{-\I t}{1+\I t}\right)^m 
\sqrt{\frac{(\alpha+1)_m}{m!}}
,\quad m\in\N_0.
\ee
\item In the case $n=1$ we have for $m\in\N$
\be\label{eq:erd2}
\E^{-\I tH_\alpha}(1,m)=
\frac {1}{(1+\I t)^{1+\alpha}}\left(\frac{-\I t}{1+\I t}\right)^{m+1}
\frac{(1+\alpha)t^2- m}{t^2}\,
\sqrt{\frac{(\alpha+2)_{m-1}}{m!}}\,.
\ee
\item
In the case $n=m$ we have
\be\label{eq:erd3}
\E^{-\I tH_\alpha}(m,m) = \frac {1}{(1+\I t)^{1+\alpha}}\left(\frac{t+\I}{t-\I}\right)^m P_m^{(\alpha,0)}\left( \frac{t^2-1}{t^2+1}\right), \quad m\in\N_0.
\ee
\end{enumerate}
\end{corollary}

\begin{proof}
Just observe
\[
P_0^{(\alpha,m)}(z) = 1, \quad P_1^{(\alpha,m-1)}(z) = -m+(m+1+\alpha)\frac{z+1}{2}.\qedhere
\]
\end{proof}

Let us also mention the following estimate.

\begin{lemma}\label{lem:B01}
If $\alpha>-1$ and $\beta+n\in\N_0$, then 
\be\label{eq:B01}
\left(\frac{1-x}{2}\right)^{{(\alpha+1)}/{2}}\left(\frac{1+x}{2}\right)^{\beta/2} \abs{P_n^{(\alpha,\beta)}(x)} \le 
\left(\frac{\Gamma(n+\alpha+1)\Gamma(n+\beta+1)}{\Gamma(n+1)\Gamma(n+\alpha+\beta+1)}\right)^{1/2}
\ee
for all $x\in [-1,1]$ and $n\in\N_0$. 
\end{lemma}

\begin{proof}
Noting that $|\E^{-\I tH_\alpha}(n,m)| \le 1$ for all $t\in \R$ since $\E^{-\I tH_\alpha}$ is a unitary group on $\ell^2$, after the change of variables 
\be\label{eq:x=t}
x=x(t):=\frac{t^2-1}{t^2+1},\quad t\in [0,\infty),
\ee
in \eqref{eq:explicit}, we arrive at \eqref{eq:B01}.
\end{proof}

\begin{remark}
The estimate \eqref{eq:B01} is of course weaker than \eqref{eq:g_unif1} (see below), however,  it holds for a larger
range of parameters. Furthermore note that Lemma \ref{lem:B01}
is also a consequence of \eqref{K3} and \eqref{K4}.
\end{remark}

It is not difficult to see that the weighted $\ell^1\to \ell^\infty$ estimates for the evolution group $\E^{-\I tH_\alpha}$ are closely connected with Bernstein-type estimates for Jacobi polynomials. Indeed, taking absolute values in \eqref{eq:explicit} we get
\be\label{eq:explicit2}
\big|\E^{-\I tH_\alpha}(n,m)\big| =
\frac{\sigma_\alpha(m)}{\sigma_\alpha(n)}\ 
\left(\frac {1}{1+t^2}\right)^{\frac{1+\alpha}{2}} \left(\frac {t^2}{1+t^2}\right)^{\frac{m-n}{2}} \abs{P_n^{(\alpha,m-n)}\left( \frac{t^2-1}{t^2+1}\right)},
\ee
for all $t\in\R$. 
With the rough inequality $t^2/(1+t^2)<1$
one immediately obtains the following estimates.

\begin{lemma}\label{lem:explest}
Let $\alpha>-1$. Then
\be\label{eq:4.07}
(1+t^2)^{\frac{1+\alpha}{2}}\big|\E^{-\I tH_\alpha}(n,m)\big| \le
\begin{cases}\sigma_\alpha(n)\sigma_\alpha(m),&\alpha\ge|m-n|,
\\[\medskipamount]\displaystyle
\frac{\sigma_\alpha(m)}{\sigma_\alpha(n)}\,\binom mn,&m-n\ge\alpha,
\\[\bigskipamount]\displaystyle
\frac{\sigma_\alpha(n)}{\sigma_\alpha(m)}\,\binom nm,&n-m\ge\alpha.
\end{cases}
\ee
for all $t\in\R$, and
\be\label{eq:4.06}
\lim_{t\to+\infty}(1+t^2)^{\frac{1+\alpha}{2}}\big|\E^{-\I tH_\alpha}(n,m)\big| =\sigma_\alpha(n)\sigma_\alpha(m)
\ee
for every fixed $n$, $m\in\N_0$.
\end{lemma}

\begin{proof}
The standard estimate \eqref{eq:Cunif} applied to \eqref{eq:explicit2} gives \eqref{eq:4.07}. Moreover, \eqref{eq:explicit2} together with \eqref{eq:normal_b} implies \eqref{eq:4.06}. 
\end{proof}

Lemma \ref{lem:explest} indicates a decay of order $O(|t|^{-(1+\alpha)})$
for $\E^{-\I tH_\alpha}(n,m)$
if one uses weighted spaces. In fact, we shall show in Section \ref{sec:V} that for $\alpha\ge 0$ the optimal weights for this decay are given by \eqref{eq:sigma_a}. Let us only record the following special cases which can be established directly from Corollary \ref{cor:cases}. 

\begin{corollary}\label{cor:cases_est}
Suppose $\alpha\ge 0$.
 \begin{enumerate}[label=(\roman*), ref=(\roman*), leftmargin=*, widest=iii] 
 \item In the case $n=0$ we have for all $m\in\N_0$
\be\label{eq:erd1b}
(1+t^2)^{\frac{1+\alpha}{2}}\big|\E^{-\I tH_\alpha}(0,m)\big|  \le \sigma_\alpha(m),\quad t\in\R.
\ee
\item In the case $n=1$ we have for all $m\in\N_0$
\be\label{eq:erd2b}
(1+t^2)^{\frac{1+\alpha}{2}}\big|\E^{-\I tH_\alpha}(1,m)\big|  \le \sigma_\alpha(1)\sigma_\alpha(m) ,\quad t\in\R.
\ee
\item In the case $n=m\in \N_0$ we have 
\be\label{eq:erd3b}
(1+t^2)^{\frac{1+\alpha}{2}}\big|\E^{-\I tH_\alpha}(m,m)\big|  \le \sigma_\alpha(m)^2,\quad t\in\R.
\ee
\end{enumerate}
\end{corollary}

\begin{proof}
(i) and (iii) are immediate from Corollary \ref{cor:cases}.
This works for (ii) as well if $\alpha\ge |m-1|$ or if $m=0$. Otherwise
we use for (ii) the new variable $x=t^2/(1+t^2)$,
so that (ii) is equivalent to
\[
\max_{x\in [0,1]} |f_m(x)| \le 1+\alpha, \quad f_m(x) = x^{\frac{m-1}{2}}\big((m+1+\alpha)x-m\big)
\]
Notice that
\[
f_1(x) = (2+\alpha)x - 1
\]
and hence
\[
\max_{x\in[0,1]} |f_1(x)| = \max(-1,1+\alpha) = 1+\alpha.
\]
For $m>1$ one computes
\[
f_m'(x) = x^{\frac{m-3}{2}}\Big(\frac{m-1}{2}\big((m+1+\alpha)x - m\big) + (m+1+\alpha)x \Big).
\]
Therefore,
\[
\max|f_m(x)| = \max(|f_m(0)|, |f_m(1)|, |f_m(x_0)|) = \max( |f_m(x_0)|, 1+\alpha)
\]
where
\[
x_0=\frac{m(m-1)}{(m+1)(m+1+\alpha)}.
\]
Moreover,
\begin{align*}
\abs{f_m(x_0)} = \frac{2m}{m+1}&\left(\frac{m(m-1)}{(m+1)(m+1+\alpha)}\right)^{\frac{m-1}{2}}\\
& < 2\left( \left(\frac{m}{m+1}\right)^{m}\left(\frac{m-1}{m}\right)^{m-1} \right)^{1/2} \le \frac{2\sqrt{2}}{3}
\end{align*}
for $m\ge 2$ since the sequence $\{(\frac{m}{m+1})^{m}\}_{m\ge 1}$ is strictly decreasing.
\end{proof}  

We finish this section with another representation for the kernel of the evolution group. Define the following functions
\be\label{eq:F2}
F_n^{(\alpha)}(t) = \frac{1}{(1/2 + \I t)^{1+\alpha}} \left(\frac{\I t-1/2}{\I t+1/2}\right)^n,
\ee
and 
\be\label{eq:G2}
G_n^{(\alpha)}(t) = \frac{1}{1/2 + \I t} \sum_{k=0}^n \binom{k+\alpha-1}{k}\left(\frac{\I t - 1/2}{\I t+ 1/2}\right)^{n-k},
\ee
for all $n\in\N_0$ and $t\in\R$.
Note that the right-hand side of
\eqref{eq:G2} involves the truncated binomial series
\cite[Section 2.5.4]{htf1}
\[
\sum_{k=0}^n\frac{(\alpha)_k}{k!}\,z^k=\hyp21{-n,\al}{-n}z.
\]
\begin{theorem}\label{th:a2}
Let $F_n^{(\alpha)}$ and $G_m^{(\alpha)}$ be given by \eqref{eq:F2} and \eqref{eq:G2}. Then
\begin{align}\label{eq:a_repr}
\E^{-\I tH_\alpha}(n,m) = (-1)^{n+m}
\frac{\sigma_\alpha(n)}{\sigma_\alpha(m)}\ \Big(F_n^{(\alpha)} * G_m^{(\alpha)}\Big)(t), 
\end{align}
where $(f\ast g)(t) = \frac{1}{2\pi}\int_{\R} f(x) g(t-x)dx$ is the convolution of $f $ and $g$.
\end{theorem}

\begin{proof}
Notice that by \cite[(4.11.28)]{erd}
\[
\frac{\Gamma(n+1)}{\Gamma(n+\alpha+1)} \int_0^\infty \E^{-\I t\lam} L_n^{(\alpha)}(\lam) \E^{-\lam/2}\lam^\alpha d\lam = \frac{1}{(1/2 + \I t)^{1+\alpha}} \left(\frac{\I t-1/2}{\I t+1/2}\right)^n  = F_n^{(\alpha)}(t),
\]
and by \cite[(4.11.27)]{erd} 
\[
\int_0^\infty \E^{-\I t\lam} L_m^{(\alpha)}(\lam) \E^{-\lam/2} d\lam = \frac{1}{1/2 + \I t} \sum_{k=0}^m \binom{k+\alpha-1}{k}\left(\frac{\I t-1/2}{\I t+1/2}\right)^{m-k}= G_m^{(\alpha)}(t).
\]
It remains to note that the Fourier transform of a product of two $L^1$ functions is equal to the convolution of their Fourier transforms.
\end{proof}

\section{Irreducible representations of $\SU(2)$ and Jacobi polynomials}\label{sec:II} 

 The theory of representations of Lie groups provides a unified point of view on the theory of basic classes of special functions. In particular, the connection between irreducible representations of the special unitary group
$\SU(2)$ and Jacobi polynomials is widely known. In this section we
give a brief account of this connection (for a detailed discussion we refer to \cite{ko}, \cite[Ch.~III]{vil}, \cite[Ch.~6]{vil-kli}).
First, recall that a group homomorphism $\varrho:\cG \to \GL(\cH)$ of a group $\cG$ into a group of all invertible linear transformation $\GL(\cH)$ on a finite dimensional complex linear space $\cH$ is called a
{\em representation} of $\cG$ (by linear operators). The dimension of $\cH$ is called the {\em degree} of the representation $\varrho$. A linear subspace $\ti\cH\subset \cH$ is called {\em invariant} with respect to the representation $\varrho$ of $\cG$ if $\varrho(g)\ti\cH \subset \ti\cH$ for all $g\in\cG$. A representation $\varrho$ is called {\em irreducible} if $\{0\}$ and $\cH$ are the only invariant subspaces.

In order to construct an irreducible representation of $\SU(2)$ of degree $d\in\N$ one needs to consider the space $\cH_d$ of homogeneous polynomials of degree $d-1$. Set $l:=(d-1)/2$.  The inner product on $\cH_d$ is defined by the requirement that the normalized monomials
\be\label{eq:basis_d}
\psi_k^d(z_1,z_2) = \binom{2l}{l-k}^{1/2}z_1^{l-k} z_2^{l+k},\quad k\in \{-l,-l+1,\dots,l-1,l\},
\ee
form an orthonormal basis.

The group $\SU(2)$ consists of all $2\times 2$ unitary matrices of determinant $1$. It is immediate to check that each $A\in \SU(2)$ has the form
\be
A=\begin{pmatrix} a & b \\ -b^\ast & a^\ast\end{pmatrix}, \quad |a|^2 + |b|^2 =1,
\ee
where $z^\ast$ denotes the complex conjugate of $z$, 
and hence $\SU(2)$ is homeomorphic to the unit sphere $\mathbb{S}^3$ in $\R^4$. Moreover, $A$ admits the following decomposition
\begin{align*}
A=&A(\phi,\theta,\varphi) = \begin{pmatrix} \cos(\theta)\, \E^{\I(\phi+\varphi)}& - \sin(\theta)\, \E^{\I(\phi-\varphi)}\\ \sin(\theta)\, \E^{-\I(\phi-\varphi)}& \cos(\theta)\,\E^{-\I(\phi+\varphi)}\end{pmatrix}\\
&=\begin{pmatrix} \E^{\I\phi} & 0 \\ 0 & \E^{-\I\phi} \end{pmatrix}\begin{pmatrix} \cos(\theta) & -\sin(\theta) \\ \sin(\theta) & \cos(\theta)\end{pmatrix} \begin{pmatrix} \E^{\I\varphi} & 0 \\ 0 & \E^{-\I\varphi} \end{pmatrix}  = A(\phi,0,0) A(0,\theta,0) A(0,0,\varphi), 
\end{align*}
where $\phi\in [0,\pi)$, $\theta\in [0,\pi/2]$ and $\varphi\in [0, \pi)$ are determined uniquely by
\[
\cos(\theta)=|a|,\quad \arg(a) = \phi+\varphi,\quad \arg(b) = \pi + \phi-\varphi,
\] 
if $ab\neq 0$. Now define a linear operator $\varrho_d(A) \in \GL(\cH_d)$ by 
\be\label{eq:varrho}
\varrho_d(A)\colon f(z_1,z_2) \mapsto f(az_1 - b^\ast z_2, bz_1+a^\ast z_2).
\ee
It is straightforward to check that $\varrho_d$ is well defined.

\begin{theorem}
The mapping  $\varrho_d\colon\SU(2) \to \GL(\cH_d)$ is an irreducible unitary representation of degree $d$ of $\SU(2)$. 
\end{theorem}

The proof of this result can be found in \cite[Section III.2.3]{vil}
(see also \cite{ko}). It turns out that the matrix representation of $\varrho_d(A)$ in the basis \eqref{eq:basis_d} (the so-called {\em Wigner d-matrix}) can be expressed by means of Jacobi polynomials. Indeed, introduce the function
\be\label{eq:g}
\g_n^{(\alpha,\beta)}(x) = \left(\frac{\Gamma(n+1)\Gamma(n+\alpha+\beta+1)}{\Gamma(n+\alpha+1)\Gamma(n+\beta+1)}\right)^{1/2}
\left(\frac{1-x}{2}\right)^{{\alpha}/{2}}\left(\frac{1+x}{2}\right)^{\beta/2} P_n^{(\alpha,\beta)}(x).
\ee
Clearly, $\big\{\g_n^{(\alpha,\beta)}\big\}_{n\in\N_0}$ is an orthogonal system in $L^2(-1,1)$ and by \eqref{eq:02}
\be
\int_{-1}^1 \big|\g_n^{(\alpha,\beta)}(x)\big|^2dx = \frac{2}{2n+\alpha+\beta+1}.
\ee 
Moreover, comparing \eqref{eq:g} with \eqref{eq:02}, we get
\be\label{eq:g=ort_p}
\g_n^{(\alpha,\beta)}(x) = \left( \frac{2}{2n+\alpha+\beta+1}\right)^{1/2}\sqrt{w^{(\alpha,\beta)}(x)}\;\p_n^{(\alpha,\beta)}(x).
\ee
Now we are ready to state the connection between $\varrho_d$ and Jacobi polynomials (see \cite[Section III.3.9]{vil}).

\begin{theorem}\label{thm:d-matrix}
Let $A=A(\phi,\theta,\varphi)\in \SU(2)$, $d\in\N$ and $\varrho_d$ be given by \eqref{eq:varrho}. Let also $l=(d-1)/2$ and $k$, $j\in\{-l,-l+1,\dots,l-1,l\}$. Then for all $k\ge 0$ and $|j|\le  k$  
\be\label{eq:d-matrix}
\varrho_d(A)_{k,j} := \big\langle \varrho_d(A)\psi_j^d,\psi_k^d\big\rangle_{\cH_d} =  
 \E^{-2\I (k\phi+j\varphi)} \g_{l-k}^{(k-j,k+j)}\big(\cos(2\theta)\big).
\ee
\end{theorem}

Since $\varrho_d(A)$ is a unitary matrix and $k\pm j\in\N_0$ in the formulation of Theorem \ref{thm:d-matrix}, we immediately conclude that 
\be\label{eq:g_unif1}
\big|\g_n^{(\alpha,\beta)}(x)\big| \le 1
\ee
for all $x\in [-1,1]$, $\alpha$, $\beta\in\N_0$ and $n\in\N_0$.
An analytic proof of a refined version of \eqref{eq:g_unif1} can be found in \cite{haa} (see inequality (20) on p.234).

\begin{lemma}[\cite{haa}]\label{lem:haa1}
\be\label{eq:g_unif2}
\big|\g_n^{(\alpha,\beta)}(x)\big| \le \left(\frac{(n+1)(n+\alpha+\beta+1)}{(n+\alpha+1)(n+\beta+1)}\right)^{1/4}
\ee
for all $x\in [-1,1]$, $\alpha$, $\beta\in\N_0$ and $n\in\N_0$.
\end{lemma}

\begin{remark}\label{rem:}
Surprisingly enough we were not able to find the estimates \eqref{eq:g_unif1} and \eqref{eq:g_unif2} for noninteger values of $\alpha$ and $\beta$ in the literature. Numerically both seem to be true for noninteger values of $\alpha$ and $\beta$.
\end{remark}

Let us also mention the  following Bernstein-type inequality obtained recently by Haagerup and Schlichtkrull in \cite{haa}.

\begin{theorem}[\cite{haa}]\label{th:haag}
There is a constant $C<12$ such that 
\be\label{eq:g_unif3}
\big|(1-x^2)^{1/4}\g_n^{(\alpha,\beta)}(x)\big| \le \frac{C}{\sqrt[4]{2n+\alpha+\beta+1}}
\ee
for all $x\in [-1,1]$, $\alpha$, $\beta\ge 0$ and $n\in\N_0$.
\end{theorem}

A few remarks are in order.

\begin{remark}
 \begin{enumerate}[label=(\roman*), ref=(\roman*), leftmargin=*, widest=iii] 
\item The optimal value for the constant $C$ in \eqref{eq:g_unif3} is not known.
\item The decay rate $n^{-1/4}$ in \eqref{eq:g_unif3} is optimal as $\alpha$ and $\beta$ tend to infinity. However, for fixed $\alpha$ and $\beta$, the decay rate is $n^{-1/2}$ as $n\to \infty$ (see \eqref{eq:Bernstein}). 
\item It was observed in \cite{haa} that \eqref{eq:g_unif3} implies the following interesting estimate for the matrix entries of $\varrho_d(A)$
\[
|\sin(2\theta)|^{1/2}\big|\varrho_d(A(\phi,\theta,\varphi))_{j,k}\big| \le Cd^{-1/4},
\]
which provides the uniform decay $d^{-1/4}$ for the matrix coefficients, where $d$ is the dimension of the representation $\varrho_d$. 
\end{enumerate}
\end{remark}

\section{Uniform weighted estimates for Jacobi polynomials}\label{sec:new} 
The main aim of this section is to prove the following inequality.

\begin{theorem}\label{th:VImain}
The Bernstein-type estimate
\be\label{eq:Bern_a0}
\left(\frac{1+x}{2}\right)^{\beta/2} \Big|P_n^{(\alpha,\beta)}\left(x\right)\Big| \le \binom{n+\alpha}{n},\quad x\in[-1,1],
\ee
holds for all $n\in\N_0$, $\beta\ge0$
and $\alpha\ge \beta - \floor{\beta}$,
where $\floor{\,.\,}$ is the usual floor function.
 
Equivalently, in terms of Meixner polynomials \eqref{K11}, we have
\be
c^{(n+x)/2} \big|M_n(x;\beta,c)\big|\le1,\qquad x\ge n,
\ee
where $0<c<1$ and $\beta\ge x-\floor{x}+1$.
\end{theorem}

The proof is based on the product formula for biangle polynomials. More precisely, let
\be\label{eq:biangle}
\cB=\{(x_1,x_2)|\, 0\le x_2^2\le x_1\le 1\}
\ee
be the parabolic biangle. Following \cite{koII,koIII}, let $R_{n}^{(\alpha,\beta)}$  denote the Jacobi polynomials normalized by $R_{n}^{(\alpha,\beta)} (1) =1$, that is,
\be\label{eq:normat1}
R_{n}^{(\alpha,\beta)}(x) = \frac{P_{n}^{(\alpha,\beta)}(x)}{P_{n}^{(\alpha,\beta)}(1)}.
\ee
For $\alpha,\beta>-1$ and $n,k\in\N_0$ such that $k\le n$, define the {\em parabolic biangle polynomials} (see, e.g., \cite[\S 2.6.1]{dunxu} and \cite[\S 3.3]{ko75}, however, with a different notation)
\be\label{eq:RmnAB}
R_{n,k}^{\alpha,\beta}(x_1,x_2)=R_{n-k}^{(\alpha,\beta+k+1/2)}(2x_1-1)\cdot x_1^{k/2} R_{k}^{(\beta,\beta)}(x_2/\sqrt{x_1}),\quad (x_1,x_2)\in \cB.
\ee

Clearly, these functions are polynomials in $x_1$ and $x_2$ of degree $n$. Moreover, for fixed $\alpha$ and $\beta$ they are orthogonal with respect to the measure
\[
(1-x_1)^{\alpha}(x_1-x_2^2)^\beta\, dx_1dx_2.
\] 
For certain values of $\alpha$ and $\beta$ the parabolic biangle polynomials have an
interpretation as spherical functions for a Gelfand pair $(K,M)$, where $K$ is a compact
group and $M$ is a closed subgroup. For these values of the parameters, the general theory of spherical
functions on Gelfand pairs yields the existence of suitable product formulas and related
hypergroup structures. The product formula in the general case was established in \cite[Thm.~2.1]{kosw}:

\begin{theorem}\label{thm:prodf-la}
Let $\alpha\ge \beta+1/2\ge 0$. Let also $0\le |x_2|\le x_1\le 1$ and $0\le |y_2|\le y_1\le 1$. If $(x_1,x_2),(y_1,y_2)\in\cB\setminus\{(0,0)\}$, then the parabolic biangle polynomials satisfy the following hypergroup-type product formula:
\begin{align}\label{eq:prodf-la}
R_{n,k}^{\alpha,\beta}(x_1^2,x_2)\cdot R_{n,k}^{\alpha,\beta}(y_1^2,y_2) = \int_{I\times J^3}R_{n,k}^{\alpha,\beta}(E^2,EG)d\mu^{\alpha,\beta}(r_1,\psi_1,\psi_2,\psi_3),
\end{align}
where $I=[0,1]$, $J=[0,\pi]$,
\begin{align*}
D &= D(x,y;r,\psi) = xy + (1-x^2)^{1/2}(1-y^2)^{1/2}r \cos \psi,\\
E &= E(x_1,y_1;r_1,\psi_1) \\
&= \big(x_1^2 y_1^2 + (1-x_1^2)(1-y_1^2)r_1^2 
+ 2x_1y_1 (1-x_1^2)^{1/2}(1-y_1^2)^{1/2}r_1 \cos \psi_1\big)^{1/2},\\
G & = 
    D\left(  \frac{D(x_1,y_1;r_1,\psi_1)}{E(x_1,y_1;r_1,\psi_1)},D(\frac{x_2}{x_1},
      \frac{y_2}{y_1};1,\psi_2); 1,\psi_3\right), 
\end{align*}
and 
\begin{align*}
&dm^\beta(\psi) =
\frac{\Gamma(\beta+\frac32)}{\Gamma(\frac12)\Gamma(\beta+1)}(\sin
\psi)^{2\beta+1} d\psi, \\
&dm^{-1}(\psi) = d\big[ \tfrac12 \delta_0(\psi) + \tfrac12
\delta_{\pi}(\psi)\big],\\
&dm^{\alpha,\beta}(r,\psi) = \frac{2 \Gamma(\alpha+1)}
     {\Gamma(\alpha-\beta)\Gamma(\beta + 1)}
     (1-r^2)^{\alpha-\beta-1}
     r^{2\beta+1} dr \,dm^{\beta-\frac12}(\psi),\\
&dm^{\alpha,\alpha}(r,\psi)  = \frac{\Gamma(\alpha+1)} {\Gamma(\alpha +
\frac12)\Gamma(\frac12)}
     (\sin \psi)^{2\alpha} \, d(\delta_1)(r)   \,d\psi,\\
&d\mu^{\alpha,\beta}(r,\psi_1,\psi_2,\psi_3) =    
     dm^{\beta-\frac12}(\psi_3) \cdot dm^{\beta-\frac12}(\psi_2)
                          \cdot   
               dm^{\alpha,\beta+\frac12}(r,\psi_1)
\end{align*}
are positive probability measures.
\end{theorem}

Before proving Theorem \ref{th:VImain}, we need 
 the following simple fact.

\begin{lemma}\label{lem:prod_est}
Let $X$ be a compact topological space and $X_0$ a dense subset of $X$. Suppose that $\phi\in C(X,\R)$
such that for each $x,y\in X_0$ there is a (positive) probability Borel measure $\mu_{x,y}$
on $X$ with the property that
\begin{equation}
\phi(x)\phi(y)=\int_X \phi(z)\,d\mu_{x,y}(z).
\label{eq:prod1}
\end{equation} 
Then 
\be
\max_{x\in X}|\phi(x)|\le 1.
\ee
\end{lemma}

\begin{proof}
Let $M:=\max_{x\in X}|\phi(x)|\ge 0$. Then, from \eqref{eq:prod1} we see
\[
\phi(x)^2 = \phi(x)\phi(x)=\int_X \phi(z)\,d\mu_{x,x}(z) \le M, \quad x\in X_0.
\]
Since $\phi$ is continuous and $X_0$ is dense in $X$, we infer $M^2\le M$, that is, $M\le 1$.
\end{proof}

\begin{remark}
In the context of hypergroups Lemma \ref{lem:prod_est}
is well known as an inequality for bounded characters on a commutative
hypergroup, see for instance the paper by Dunkl \cite[Prop.~2.2(2)]{d73}.
\end{remark}

\begin{proof}[Proof of Theorem \ref{th:VImain}] Using the product formula and Lemma \ref{lem:prod_est}, we immediately conclude that
\be
\big|R_{n,k}^{\alpha,\beta}(x_1^2,x_2)\big|\le 1, \quad (x_1,x_2)\in\cB,
\ee
for all $k\le n$ and $\alpha\ge \beta+1/2\ge 0$. By \eqref{eq:Cunif} and \eqref{eq:normat1} we know
\[
\big|R_{k}^{(\beta,\beta)}(x_2/\sqrt{x_1})\big| \le 1, \quad (x_1,x_2)\in\cB\setminus\{(0,0)\}
\]
for all $k\in\N_0$ and $\beta\ge -1/2$, and hence we conclude (replacing $\beta+1/2$ by $\beta$)
\be
\left(\frac{x+1}{2}\right)^{k/2}\abs{P_{n-k}^{(\alpha,\beta+k)}(x)} \le \binom{n-k+\alpha}{n-k},\quad x\in[-1,1],
\ee
for all $k\le n$ and $\alpha\ge \beta\ge 0$.  Since $n\in \N_0$ is arbitrary, we can replace $n-k$ by $n\in\N_0$. Moreover, choosing $\beta\in [0,1)$ and noting that $k\in\N_0$ is arbitrary, we finally end up with
\[
\left(\frac{x+1}{2}\right)^{\floor{\beta}/2}\abs{P_{n}^{(\alpha,\beta)}(x)} \le \binom{n+\alpha}{n},\quad x\in[-1,1],
\]
which holds for all $n\in\N_0$ and $\alpha\ge \beta-\floor{\beta}$. 
Since $(x+1)/2\le 1$ for all $x\in[-1,1]$, this completes the proof.
\end{proof}

\begin{remark}
 In fact, \cite[Thm.~2.1]{kosw} gives \eqref{eq:prodf-la} also for
the case $(x_1,x_2)=(0,0)$ or $(y_1,y_2)=(0,0)$, with a somewhat simpler
measure on the right-hand side. Hence we might have worked with a version
of Lemma \ref{lem:prod_est} without the restriction to a dense subset.
\end{remark}
 
We would like to finish this section with the following remark. 
We have two more proofs of Theorem \ref{th:VImain}\iflong{} (see Appendix~\ref{secapp})\fi, however, for a smaller set of parameters $\alpha$ and $\beta$. More precisely, using the addition formula for disk polynomials \cite{koIII}, one can prove \eqref{eq:Bern_a0} for all $\alpha\ge 0$ and $\beta\in\N_0$. The third proof is based on \eqref{eq:g_unif1} and hence inherits the restriction $\alpha$ and $\beta\in \N_0$. It uses the Sonin--P\'olya theorem\footnote{In the literature Sonin is also written as Sonine.} \cite[footnote to Theorem 7.31.1]{sz} and leads to the following result:

 \begin{theorem}\label{th:5.12}
Inequality \eqref{eq:Bern_a0} holds for all indices $\alpha$, $\beta$ for which \eqref{eq:g_unif1} holds.
\end{theorem}

We firmly expect that \eqref{eq:g_unif1} holds for all $x\in [-1,1]$ and $\alpha,\beta\ge 0$, which in particular would imply \eqref{eq:Bern_a0} for all $\alpha,\beta\ge 0$.

\section{Dispersion estimates for the evolution group $\E^{-\I tH_\alpha}$}\label{sec:V}

It turns out that Theorem \ref{thm:explicit} (see also \eqref{eq:explicit2}) establishes a connection between Bernstein-type inequalities and dispersion estimates for the discrete Laguerre operators $H_\alpha$.  
In this section we shall present some $\ell^1\to \ell^\infty$ decay estimates for the evolution group $\E^{-\I tH_\alpha}$ based on Bernstein-type inequalities from the previous sections. 

First, notice that \eqref{eq:explicit2} can be rewritten in terms of the function $\g_n^{(\alpha,\beta)}$ introduced in \eqref{eq:g}:
\begin{align}\label{eq:4.11}
\big|\E^{-\I tH_\alpha}(n,m)\big|
	= \frac {1}{\sqrt{1+ t^2}} \left|\g_n^{(\alpha,m-n)}\left(\frac{t^2-1}{t^2+1}\right)\right|,\quad m\ge n.
\end{align}
Hence the estimate \eqref{eq:g_unif1} immediately implies

\begin{theorem}\label{thm:decay1}
Let $\alpha\in\N_0$. Then the following estimate
\be\label{eq:decay}
\|\E^{-\I tH_\alpha}\|_{\ell^1\to \ell^\infty} \le \frac{1}{\sqrt{1+t^2}},\quad t\in\R,
\ee
holds. Moreover, in the case $\alpha=0$, the inequality can be replaced by equality.
\end{theorem}

\begin{proof}
To prove the last claim it suffices to note that 
\[
\|\E^{-\I tH_0}\|_{\ell^1\to \ell^\infty} \ge \big|\E^{-\I tH_0}(0,0)\big| = \frac {1}{\sqrt{1+ t^2}}
\]
for all $t\in\R$.
\end{proof}

\begin{remark}
\begin{enumerate}[label=(\roman*), ref=(\roman*), leftmargin=*, widest=ii]
\item
The case $\alpha=0$ was proven in \cite{kt}. Using a different approach, a weaker estimate in the case $\alpha=0$ was obtained in \cite{ks15b}.
\item
Using Lemma~\ref{lem:haa1} (see also \eqref{eq:g_unif1}), we get the somewhat stronger estimate
\begin{align*}
\big|\E^{-\I tH_\alpha}(n,m)\big| \le \frac {1}{\sqrt{1+ t^2}} 
 \left(\frac{(n+1)(m+\alpha+1)}{(m+1)(n+\alpha+1)}\right)^{1/4},
	\end{align*}	
which holds for all $m\ge n$, $\alpha\in\N_0$ and $t\in\R$.
\end{enumerate}
\end{remark}

\begin{conjecture}\label{con:03}
We conjecture that \eqref{eq:g_unif1} as well as Lemma~\ref{lem:haa1} hold true for all $\alpha,\beta\ge 0$ and consequently
Theorem~\ref{thm:decay1} holds for all $\alpha\ge0$.
\end{conjecture}

Applying the Haagerup--Schlichtkrull inequality \eqref{eq:g_unif3} to \eqref{eq:4.11} we obtain another estimate, which holds for all $\alpha\ge 0$:

\begin{theorem}\label{thm:decay2}
Let $\alpha\ge 0$. There is a positive constant $C<6\sqrt{2}$ such that the following inequality
\be\label{eq:decay_c}
\big|\E^{-\I tH_\alpha}(n,m)\big| \le  \frac{C|t|^{-1/2}}{\sqrt[4]{n+m+\alpha+1}},
\ee	
holds for all $n$, $m\in\N_0$ and $t\neq 0$.
\end{theorem}


\begin{remark}
\begin{enumerate}[label=(\roman*), ref=(\roman*), leftmargin=*, widest=iii] 
\item The estimate in Theorem \ref{thm:decay2} provides only a $t^{-1/2}$ decay, however, it gives an $(n+m)^{-1/4}$ decay of the matrix coefficients.
\item Let us also mention that the Erd\'elyi--Magnus--Nevai conjecture \eqref{eq:emn_con} would imply the following estimate
\be
\big|\E^{-\I tH_\alpha}(n,m)\big| \le \frac{C}{\sqrt{|t|}}\frac{(\alpha+|m-n|)^{1/4}}{(n+m+\alpha+1)^{1/2}},\quad t\neq 0. 
\ee 
The latter shows that on diagonals, i.e., when $m-n=const$, the decay of the matrix elements is $n^{-1/2}$ as $n\to \infty$.
However, it does not improve \eqref{eq:decay_c} when $m-n$ tends to infinity. 
\end{enumerate}
\end{remark}

The estimates \eqref{eq:decay} and \eqref{eq:decay_c} provide a non-integrable decay as $t\to\infty$. However, in order to establish stability for soliton type solutions to nonlinear equations it is desirable to have an integrable decay in $t$. As we mentioned in Section \ref{sec:IV}, we expect a decay of order $O(|t|^{-(1+\alpha)})$, however, in weighted spaces. To this end
let $\sigma = \{\sigma(n)\}_{n\ge 0}$ be a positive sequence. Consider the weighted $\ell^p$ spaces equipped with the norm
\begin{equation*}
   \Vert u\Vert_{\ell^p({\sigma})}= \begin{cases} \left( \sum_{n\in\N_0} \sigma(n) |u(n)|^p\right)^{1/p}, & \quad p\in[1,\infty),\\
   \sup_{n\in\N_0} \sigma(n) |u(n)|, & \quad p=\infty. \end{cases}
\end{equation*}
Of course, the case $\sigma\equiv 1$ corresponds to the usual $\ell^p(\sigma)=\ell^p$ spaces without weight. Specifically
we will work with the weights $\sigma_\alpha(n)$, given in
\eqref{eq:sigma_a},
and consider the weighted spaces $\ell^1({\sigma_\alpha})$ and $\ell^\infty({\sigma^{-1}_\alpha})$. Notice that
\be\label{eq:sigmassymp}
\sigma_\alpha(n) = \frac{n^{\alpha/2}}{\sqrt{\Gamma(\alpha+1)}}(1+o(1)),\quad n\to \infty.
\ee 

\begin{theorem}\label{thm:decay_a}
The following equality
\be\label{eq:decay_a}
\|\E^{-\I tH_\alpha}\|_{\ell^1(\sigma_\alpha)\to \ell^\infty(\sigma_\alpha^{-1})} = \left(\frac{1}{1+t^2}\right)^{\frac{1+\alpha}{2}},\quad t\in\R,
\ee
holds for all $\alpha\ge 0$.
\end{theorem}

\begin{proof}
First of all, noting that $\E^{-\I tH_\alpha}(0,0) = (1+\I t)^{-1-\alpha}$ (see Corollary \ref{cor:cases}(i)), we get
\[
\|\E^{-\I tH_\alpha}\|_{\ell^1(\sigma_\alpha)\to \ell^\infty(\sigma_\alpha^{-1})} \ge \left(\frac{1}{1+t^2}\right)^{\frac{1+\alpha}{2}},\quad t\in\R.
\]
The converse inequality 
\be\label{eq:4.08}
\frac{\big|\E^{-\I tH_\alpha}(n,m)\big|}{\sigma_\alpha(n)\sigma_\alpha(m)} \le \left(\frac {1}{1+t^2}\right)^{\frac{1+\alpha}{2}} ,\quad t\in\R,\,\, n,m\in\N_0, 
\ee
follows from the Bernstein-type estimate \eqref{eq:Bern_a0}. Indeed, by \eqref{eq:explicit}, it suffices to consider the case $n\le m$. Using \eqref{eq:explicit2} and making the change of variables \eqref{eq:x=t}, we get
\be\label{eq:4.09}
\left({1+t^2}\right)^{\frac{1+\alpha}{2}}\frac{\abs{\E^{-\I tH_\alpha}(n,m)}}{\sigma_\alpha(n)\sigma_\alpha(m)} = \binom{n+\alpha}{n}^{-1} 
 \left(\frac {1+x}{2}\right)^{\frac{m-n}{2}}\abs{P_n^{(\alpha,m-n)}(x)}.
\ee
However, by \eqref{eq:Bern_a0}, the right-hand side is less than $1$, which completes the proof.
\end{proof}

\begin{remark}
An inspection of $\E^{-\I tH_\alpha}(n,m)$ with $n=m=1$ (see the proof of Corollary \ref{cor:cases_est}) shows that \eqref{eq:Bern_a0} is no longer true for $\alpha <0$. However, 
we expect that the following estimate
\[
\|\E^{-\I tH_\alpha}\|_{\ell^1\to \ell^\infty} = \OO(|t|^{-(1+\alpha)}),\quad t\to\infty,
\]
holds true for all $\alpha\in(-1,0)$.
\end{remark}

\section{Conclusions}\label{sec:VII}

\subsection{A hunt for Bernstein-type inequalities} 

The main aim of this paper was to prove dispersive decay for the evolution group $\E^{-\I tH_\alpha}$. It turned out that this problem is closely related to Bernstein-type inequalities for \eqref{eq:Bern_ab} and, in particular, has led us to new Bernstein-type inequalities \eqref{eq:Bern_a0} and \eqref{eq:B01}. In fact, the search for an optimal decay in $t$ or in $m$ and $n$ for the kernel $\E^{-\I tH_\alpha}(n,m)$ leads to a wider class of Bernstein-type inequalities.
More precisely, recall the change of variables \eqref{eq:x=t} and let $\eta\in [0, {1+\alpha}]$, $\nu\ge 0$ be fixed.
Then \eqref{eq:4.09}, after substitution of \eqref{eq:sigma_a}, can be rewritten as
\begin{align}\label{eq:explicit3}
\begin{split}
&(1+t^2)^{\frac{\eta}{2}}\big|\E^{-\I tH_\alpha}(n,m)\big| \left(\frac{t^2}{1+t^2}\right)^{\frac{\nu}{2}} =\\
&\quad  \frac{\sigma_\alpha(m)}{\sigma_\alpha(n)}\,
 \left(\frac{1+x}{2}\right)^{\frac{m-n+\nu}{2}} \left(\frac{1-x}{2}\right)^{\frac{1+\alpha-\eta}{2}} \abs{P_n^{(\alpha,m-n)}\left(x\right)},
\end{split}
\end{align}
for all $n\le m$. Let $\sigma = \{\sigma(n)\}_{n\ge 0}$ be a positive weight. Noting that
\[
\|\E^{-\I tH_\alpha}\|_{\ell^1(\sigma)\to \ell^\infty(\sigma^{-1})} = \sup_{n,m\in\N_0} \sigma(n)^{-1}\abs{\E^{-\I tH_\alpha}(n,m)} \sigma(m)^{-1},
\]
we conclude that the dispersive decay estimate
\be\label{eq:decay_eta}
\|\E^{-\I tH_\alpha}\|_{\ell^1(\sigma)\to \ell^\infty(\sigma^{-1})} \le C(1+t^2)^{-\eta/2},\quad t\in\R,
\ee
would follow from the Bernstein-type bound
\begin{equation}
\left(\frac{1-x}{2}\right)^{\frac{1+\alpha-\eta}{2}} \left(\frac{1+x}{2}\right)^{\frac{m-n+\nu}{2}}
\abs{P_n^{(\alpha,m-n)}\left(x\right)}\le
C \sigma(n)\sigma(m)
\sqrt{\frac{(\alpha+1)_n\,m!}{(\alpha+1)_m\,n!}}
\end{equation}
for all $n\le m$ and $x\in(-1,1)$. 
Clearly, the latter is a uniform weighted estimate for \eqref{eq:Bern_ab} with $a=\frac{1+\alpha-\eta}{2}$ and $b=\frac{\beta+\nu}{2}$.
In this respect let us mention that our Theorem~\ref{th:VImain} gives rise to $\eta=1+\alpha$ and $\nu=0$; the estimates \eqref{eq:g_unif1}--\eqref{eq:g_unif2} correspond to the case $\eta=1$ and $\nu=0$; the Erdelyi--Magnus--Nevai conjecture \eqref{eq:emn_con} and the Haagerup--Schlichtkrull inequality \eqref{eq:g_unif3} correspond to $\eta=\nu=1/2$. 
 
\subsection{1-D spherical Schr\"odinger operators} Let us finish this paper by comparing our results with the recent study of dispersive estimates for the one-dimensional spherical Schr\"odinger operators
\[
{\rm H}_l = -\frac{d^2}{dx^2} + \frac{l(l+1)}{x^2},\quad l\ge -\frac{1}{2},
\]  
acting in $L^2(\R_+)$ (${\rm H}_l$ denotes the Friedrichs extension if $l\in (-1/2,1/2)$).
 In the free case $l=0$, one has 
 \[
 \|\E^{-\I t{\rm H}_0}\|_{L^1(\R_+)\to L^\infty(\R_+)} = \OO(|t|^{-1/2}),\quad t\to \infty.
 \] 
 It was shown in \cite{kotr} (see also \cite{ktt}) that $\|\E^{-\I t{\rm H}_l}\|_{L^1\to L^\infty} = \OO(|t|^{-1/2})$ as $t\to\infty$ for all $l\ge -1/2$. On the other hand, considering weighted $L^1\to L^\infty$ estimates, one can improve the decay in $t$ for positive $l>0$ \cite{ktt,kotr}: 
\[
\norm{\E^{-\I t {\rm H}_{l}}}_{L^1(\R_+;x^l) \to L^\infty(\R_+;x^{-l})}
= \mathcal{O}(|t|^{-l-1/2}),\quad t\to\infty.
\]
Since $\alpha$ in \eqref{eq:H0} can be seen as a measure of the delocalization of the field configuration and it is related to the planar angular momentum \cite{a13}, our dispersive decay  estimates \eqref{eq:decay} and \eqref{eq:decay_a} can be viewed as analogues of the above mentioned results for spherical Schr\"odinger operators from \cite{ktt,kotr}.

\appendix

\iflong

\section{Two alternate proof for the main theorem}
\label{secapp}

\subsection{The addition formula for disk polynomials}\label{sec:II.02}

Following \cite{koII,koIII}, let $R_{n}^{(\alpha,\beta)}$ denote the Jacobi polynomials normalized as in \eqref{eq:normat1}.
Consider the {\em disk polynomials} (see, e.g., \cite[\S 2.6.3]{dunxu} and \cite[\S 3.3]{ko75}, however, with a different notation)
\be\label{eq:RmnA}
R_{m,n}^{(\alpha)}(r\E^{\I\phi})=r^{|m-n|}\E^{\I(m-n)\phi}R_{\min(m,n)}^{(\alpha,|m-n|)}\big(2r^2-1\big).
\ee
For $\alpha=q-2$ with an integer $q\ge 2$ and under a suitable choice of coordinates on the unit sphere $\mathbb{S}^{2q}$ in $\C^q$, these functions are zonal surface harmonics of type $(m,n)$ as introduced by Ikeda \cite{ik}.
This interpretation of disk polynomials was the key to the following addition formula established in \cite{koII,koIII}.

\begin{theorem}[\cite{koIII}] 
Let $\alpha>0$. The following addition formula holds:
\begin{align}
R_{m,n}^{(\alpha)}&\big(\cos(\theta_1)\E^{\I\phi_1}\cos(\theta_2)\E^{\I\phi_2} + \sin(\theta_1)\sin(\theta_2)r\E^{\I\psi}\big) \nn \\
&= \sum_{k=0}^m\sum_{l=0}^n c_{m,n,k,l}^{(\alpha)}(\sin(\theta_1))^{k+l}R_{m-k,n-l}^{(\alpha+k+l)}\big(\cos(\theta_1)\E^{\I\phi_1}\big) \label{eq:addf-la}\\
&\qquad \times (\sin(\theta_2))^{k+l}R_{m-k,n-l}^{(\alpha+k+l)}\big(\cos(\theta_2)\E^{\I\phi_2}\big)R_{k,l}^{(\alpha-1)}(r\E^{\I\psi}),\nn
\end{align}
where
\be
c_{m,n,k,l}^{(\alpha)} = \frac{\alpha}{\alpha+k+l}\binom{m}{k}\binom{n}{l}\frac{(\alpha+n+1)_k(\alpha+m+1)_l}{(\alpha+l)_k(\alpha+k)_l}.
\ee
\end{theorem}

The addition formula \eqref{eq:addf-la} leads to \eqref{eq:Bern_a0} for integer $\beta$.

\begin{proof}[Proof of Theorem \ref{th:VImain}: The case $\alpha> 0$, $\beta\in\N_0$]
Setting $\theta_1=\theta_2=\theta\in[0,\pi]$,
$\phi_1=\phi_2=\psi=0$, $r=1$ in \eqref{eq:addf-la}
and assuming $n\le m$, we end up with
 \begin{align}
1=R_{m,n}^{(\alpha)}&\big(1\big) = \sum_{k=0}^m\sum_{l=0}^n c_{m,n,k,l}^{(\alpha)}(\sin(\theta))^{2(k+l)}R_{m-k,n-l}^{(\alpha+k+l)}\big(\cos(\theta)\big)^2.
\end{align}
In particular, since all summands are nonnegative and using the normalization, we easily get the following estimate (notice that $c_{m,n,0,0}^{(\alpha)}=1$)
\be\label{eq:5.07}
R_{m,n}^{(\alpha)}\big(\cos(\theta)\big)^2 \le 1,\quad \theta \in [0,\pi],
\ee
which proves the claim after a simple change of variables.
Then the case $\alpha=0$, $\beta\in\N_0$ follows by continuity.
\end{proof}

\begin{remark}
A different proof of \eqref{eq:5.07} can be found in \cite[Proposition 2.6.7(i)]{dunxu}.
\end{remark}

\subsection{The Sonin--Poly\'a theorem}

Here we provide a direct proof for Theorem~\ref{th:VImain} using the
Sonin--P\'olya theorem and the inequality \eqref{eq:g_unif1},
which gives some further insight into the behavior of the left-hand
side of \eqref{eq:Bern_a0}.

We divide the proof in three steps. First, let us establish an explicit neighborhood of $x=1$ where \eqref{eq:Bern_a0} holds.
For this we recall the Sonin--P\'olya theorem \cite[footnote to Theorem 7.31.1]{sz}, which associates with
a solution $y$ of a differential equation

\begin{equation}\label{K10}
(py')'+q y=0
\end{equation}
a {\em Sonin function}

\begin{equation}\label{K8}
S(x):=y(x)^2+\frac{p(x)}{q(x)}\,y'(x)^2
\end{equation}
and then we observe that 
\[
S'(x)=-(p q)'(x)\,\left(\frac{y'(x)}{q(x)}\right)^2,
\] 
by which successive relative maxima of
$y^2$ form an increasing or decreasing sequence according as
$pq$ decreases or increases on the corresponding interval.

\begin{lemma}\label{lem:6.04}
Let $\alpha$, $\beta\ge 0$ and $n\in\N_0$.
Put $\lambda_n = n(n+\alpha+\beta+1)$.
There are points
\be
-1 \le x_0 \le x_1 \le 1
\label{K9}
\ee
given explicitly by
\[
x_0 = -1+\frac{2\beta^2}{\beta^2+2\beta(1+\alpha)+4\lambda_n},
\quad
x_1= 1-\frac{2(1+2\alpha) (\beta + \beta  \alpha + 2 \lambda_n)}{(1+\alpha) \big(\beta^2 + 2\beta(1+\alpha) +4\lambda_n \big)},
\]
 such that the relative maxima of
\[
\left(\frac{1+x}{2}\right)^\beta \big|P_n^{(\alpha,\beta)}\left(x\right)\big|^2
\]
are increasing on $(x_1,1]$ and decreasing on $(x_0,x_1)$. Moreover,
inequality~\eqref{eq:Bern_a0} holds on $[x_1,1]$ and there are no relative
maxima on $[-1,x_0)$.
\end{lemma}

\begin{proof}
Abbreviate
\[
y(x):= \left(\frac{1+x}{2}\right)^{\beta/2} P_n^{(\alpha,\beta)}
\left(x\right).
\]
Then rewriting of the differential equation \cite[(4.2.1)]{sz} for
Jacobi polynomials shows that $y$ satisfies \eqref{K10} with
\[
p(x)=(1-x)^{\al+1} (1+x)\quad{\rm and}\quad
4 p(x)q(x)=f(x) (1-x)^{2\al+1},
\]
where
\[
f(x)=-\beta^2 +2\beta(1 + \alpha)+4 \lambda_n +
x \big(\beta^2 + 2\beta(1+ \alpha) + 4\lambda_n\big).
\]
The corresponding Sonin function $S$ given by \eqref{K8}
then has a singularity
at the zero $x_0$ of $f$ such that $S(x)\to-\infty$ or $+\infty$
according as $x\uparrow x_0$ or $x\downarrow x_0$.
Then a calculation shows that
\[
4(1-x)^{-2\alpha} (pq)'(x)=
\big(\beta^2 + 2\beta(1+ \alpha) + 4\lambda_n\big)(1-x)
-(2\al+1)f(x),
\]
which has a zero at $x_1$. The inequalities \eqref{K9} are easily
checked. Now we see that
$S(x)$ decreases from $0$ to $-\infty$ on $[-1,x_0)$, decreases from
$\infty$ to $S(x_1)$ on $(x_0,x_1]$
and increases from $S(x_1)$ to $S(1)=P_n^{(\alpha,\beta)}(1)^2$ on
$[x_1,1]$.
In particular, $S(x)<0$ on $(-1,x_0)$ and hence there cannot be
any maxima of $y(x)^2$ in this interval.
\end{proof}

Now let us find an explicit neighborhood of $x=-1$ where \eqref{eq:Bern_a0} holds. 

\begin{lemma}\label{lem:6.05}
Inequality \eqref{eq:Bern_a0} holds on the interval
\[
[-1,x_2), \qquad x_2= 1-2 \left(\binom{n+\alpha}{n}\binom{n+\alpha+\beta}{n+\beta}\right)^{-1/\alpha}
\]
for all indices $\beta,\alpha$ for which \eqref{eq:g_unif1} holds.
\end{lemma}

\begin{proof}
This follows upon inserting \eqref{eq:g_unif1} into the desired inequality \eqref{eq:Bern_a0} and solving for $x$.
\end{proof}

We need the following technical lemma, which allows to estimate $x_2$.

\begin{lemma}\label{lem:binom}
The following inequality for binomial coefficients holds for all $x\ge 0$:
\be\label{eq:binom_est}
\binom{x+y}{x} \ge \begin{cases} (x+y)^y, & 0\le y \le 1,\\  (\frac{x+y}{y})^y & 1 \le y.\end{cases}
\ee
\end{lemma}

\begin{remark}
The case $0\le y \le 1$ follows from an inequality due to Wendel \cite[eqn.\ (7)]{wen}. The case $1 \le y$ can be found in, e.g., \cite[eqn.\ (41)]{qncc} (as written there after the formula: {\em the reverse inequality (41) holds if $a>1$}).
\end{remark}

Finally we note:

\begin{lemma}\label{lem:6.06}
Suppose $\alpha\ge1$ and $n\ge1$ or $\alpha\ge0$ and $n\ge2$. Then
\[
x_1 \le x_2.
\]
\end{lemma}

\begin{proof}
First of all note that $x_1 \le x_2$ is equivalent to
\[
\left(\binom{n+\alpha}{n}\binom{n+\alpha+\beta}{n+\beta}\right)^{-1/\alpha} \le \frac{(1+2\alpha) (\beta(1+  \alpha) + 2 \lambda_n)}{(1+\alpha) \big(\beta^2 + 2\beta(1+\alpha) +4\lambda_n \big)}.
\]
Now inequality \eqref{eq:binom_est} for $\alpha\ge1$ implies that $x_1 \le x_2$ will hold if
\[
\frac{\alpha}{n+\alpha} \frac{\alpha}{n+\alpha+\beta}\le \frac{(1+2\alpha) (\beta(1+  \alpha) + 2 \lambda_n)}{(1+\alpha) \big(\beta^2 + 2\beta(1+\alpha) +4\lambda_n \big)}.
\]
However, it is easy to check that a stronger inequality (note that $1+\alpha\le n+\alpha$ for $n\ge 1$) 
\[
\alpha\big(\beta^2 + 2\beta(1+\alpha) +4\lambda_n \big) \le 2(\beta(1+  \alpha) + 2 \lambda_n)(n+\alpha+\beta)
\]
holds true for $n \ge 1$ and $\beta,\alpha\ge 0$. The case $0\le \alpha\le1$ is similar.
\end{proof}

Combining Lemmas \ref{lem:6.04}, \ref{lem:6.05} and \ref{lem:6.06}, we arrive at Theorem \ref{th:5.12}:

\begin{proof}[Proof of Theorem~\ref{th:5.12}]
Since the cases $n=0,1$ can be checked directly (see Corollary~\ref{cor:cases_est}), combining Lemmas \ref{lem:6.04}, \ref{lem:6.05} and \ref{lem:6.06}, we
conclude that \eqref{eq:Bern_a0} holds for all $\alpha,\beta\ge0$ for which \eqref{eq:g_unif1} holds and hence in particular for $\alpha,\beta\in\N_0$.
\end{proof}
\fi

\bigskip
{\bf Acknowledgments.}
We are indebted to Alexander Aptekarev, Christian Krattenthaler, Margit R\"osler, Walter Van Assche and Michael Voit for discussions on these topics.
We also thank the referees for the careful reading of our manuscript.


\end{document}